\documentclass[11pt]{article}
\usepackage{amsthm, amsmath, amssymb, amsfonts, url, booktabs, tikz, setspace, fancyhdr, bm}
\usepackage{geometry}
\usepackage{hyperref, enumerate}
\usepackage[shortlabels]{enumitem}
\usepackage[babel]{microtype}
\usepackage[english]{babel}
\usepackage[capitalise]{cleveref}
\usepackage{comment}
\usepackage{bbm}
\usepackage{csquotes}
\usepackage{mathabx}
\usepackage{tikz}
\usepackage{graphicx}
\usepackage{float}
\usepackage{xcolor}
\usepackage{mathtools}
\usepackage{mathrsfs}

\usetikzlibrary{positioning, arrows.meta, shapes.geometric}

\counterwithin{figure}{section}

\newtheorem{theorem}{Theorem}[section]

\newtheorem{lemma}[theorem]{Lemma}

\newtheorem{claim}[theorem]{Claim}

\usetikzlibrary{decorations.pathmorphing}
\newcommand{\SO}{\operatorname{SO}}
\theoremstyle{definition}

\newtheorem*{defn-non}{Definition}

\newlist{Case}{enumerate}{2}
\setlist[Case, 1]{%
    label           =   {\bfseries Case \arabic*.},
    labelindent=1em ,labelwidth=1.3cm, labelsep*=1em, leftmargin =!
}
\setlist[Case, 2]{%
    label           =   {\bfseries Subcase \arabic{Casei}.\arabic*.},
    labelindent=-1em ,labelwidth=1.3cm, labelsep*=1em, leftmargin =!
}

\newenvironment{poc}{\begin{proof}[Proof of the claim]}{\end{proof}}

\usepackage{todonotes}

\title{The sharp asymptotic density of zero-sum-free spherical sets} 
\author{
Hong-Jun Ge\thanks{School of Mathematical Sciences, University of Science and Technology of China,
Hefei, China. Email: gehj22@mail.ustc.edu.cn.}
\and 
Zixiang Xu\thanks{School of Mathematical Sciences, Zhejiang University, Hangzhou, China. Email: zixiangxu@zju.edu.cn.}
}

\date{}
\begin{document}
\maketitle
\begin{abstract}
A measurable set \(A\subseteq \mathbb S^{d-1}\) is called zero-sum-free if
there are no \(\boldsymbol{x},\boldsymbol{y},\boldsymbol{z}\in A\) with
\(\boldsymbol{x}+\boldsymbol{y}+\boldsymbol{z}=\boldsymbol{0}\). Bukh asked
whether every zero-sum-free measurable subset of \(\mathbb S^{d-1}\), for
\(d\ge3\), has normalized surface measure at most \(\frac{1}{2}\). He also
pointed out that even the asymptotic behavior as \(d\to\infty\) was
unknown. We answer Bukh's
asymptotic question by proving that every such set has normalized surface
measure at most
\(
    \frac{\lfloor (d+1)^2/2\rfloor}{d(d+1)}
    =
    \frac{1}{2}+O\left(\frac{1}{d}\right).
\)
Since the lower bound \(\frac{1}{2}\) comes from open hemispheres, this
determines the asymptotic extremal density.

By monotonicity, upper bounds in low-dimensional cases are especially important.
We use a stability argument to improve the bound from
\(\frac{3}{5}\) to \(\frac{71}{120}\) in dimensions \(4\) and \(5\). 
\end{abstract}
\section{Introduction}
Let \(\mathbb S^{d-1}\) be the unit sphere in \(\mathbb R^d\), equipped
with its rotation-invariant probability measure \(\sigma_{d-1}\). Then
\(\sigma_{d-1}(A)\) is the probability that a uniformly random point of
\(\mathbb S^{d-1}\) lies in \(A\).  A measurable set
\(A\subseteq\mathbb S^{d-1}\) is called zero-sum-free if there are no
\(\boldsymbol{x},\boldsymbol{y},\boldsymbol{z}\in A\) such that
\(  \boldsymbol{x}+\boldsymbol{y}+\boldsymbol{z}=\boldsymbol{0}.
\)
For \(d\ge2\), define
\[
    m_d:=
    \sup\left\{\sigma_{d-1}(A):
    A\subseteq\mathbb S^{d-1}\text{ is measurable and zero-sum-free}
    \right\}.
\]

There is an immediate lower bound \(m_d\ge\frac{1}{2}\).  Indeed, for any
\(\boldsymbol{u}\in\mathbb S^{d-1}\), the open hemisphere
\[
    A_{\boldsymbol{u}}
    :=
    \{\boldsymbol{x}\in\mathbb S^{d-1}:
    \boldsymbol{x}\cdot\boldsymbol{u}>0\}
\]
has measure \(\frac{1}{2}\), and if
\(\boldsymbol{x},\boldsymbol{y},\boldsymbol{z}\in A_{\boldsymbol{u}}\),
then
\(  (\boldsymbol{x}+\boldsymbol{y}+\boldsymbol{z})\cdot\boldsymbol{u}>0.
\)
Thus \(A_{\boldsymbol{u}}\) contains no zero-sum triple. Bukh asked whether this simple construction is always extremal for
\(d\ge3\), that is, whether \(m_d=\frac{1}{2}\) for every \(d\ge3\).  He
also pointed out that, for spheres, it was not even known whether
\(m_d\to\frac{1}{2}\) as \(d\to\infty\)~\cite{BukhProblems}.

This question is closely connected with some of the most basic themes in
additive combinatorics.  The equation \(x+y+z=0\) is a fundamental
three-term additive relation.  Classical sum-free sets avoid solutions to
\(x+y=z\), while cap sets in \(\mathbb F_3^n\) avoid nontrivial
three-term arithmetic progressions, equivalently nontrivial solutions to
\(x+y+z=0\).  The cap set problem has played a major role in modern
additive combinatorics, the breakthrough polynomial method of
Croot--Lev--Pach and Ellenberg--Gijswijt gave exponentially strong upper
bounds for progression-free subsets of finite vector spaces
\cite{CrootLevPach2017,EllenbergGijswijt2017}. 
Bukh's problem asks for
an analogous extremal question for the same formal additive relation  \(x+y+z=0\), but in
a very different ambient space: the finite vector space is replaced by
the unit sphere.  
Thus the usual translation symmetry of additive
combinatorics is no longer available, and the relevant structure comes
instead from rotations and finite configurations on the sphere.

This spherical setting also gives the problem a very concrete geometric
interpretation.  If
\(\boldsymbol{x},\boldsymbol{y},\boldsymbol{z}\in\mathbb S^{d-1}\) satisfy
\(\boldsymbol{x}+\boldsymbol{y}+\boldsymbol{z}=\boldsymbol{0},
\)
then they form a centered regular triangle on the sphere, indeed,
their pairwise inner products are all equal to \(-\frac{1}{2}\).  Hence
Bukh's question is the problem of determining how large a measurable
spherical set can be if it avoids centered regular triangles.  In this
form, it fits naturally into the broader geometric tradition of studying
measurable sets with forbidden configurations, such as sets with excluded
distances, geometric independence densities, spherical sets avoiding
specified angles, and related problems in Euclidean Ramsey theory and
harmonic analysis, see, for example,
\cite{Bukh2008,CastroSilva2023,DeCortePikhurko2018}.  It is also
connected to the study of finite configurations on spheres, including
spherical codes and designs, beginning with the work of
Delsarte--Goethals--Seidel and its many developments
\cite{BondarenkoRadchenkoViazovska2013,DelsarteGoethalsSeidel1977}.

The restriction \(d\ge3\) is essential.  
In the circle case,
corresponding to \(d=2\), the zero-sum triples are exactly the triples of
points with angles separated by \(2\pi/3\), and a simple orbit argument
gives the sharp value \(m_2=\frac{2}{3}\).  
Thus Bukh's question asks
whether a new rigidity phenomenon appears once one passes from circle to spheres of dimension
at least two.

Our first result answers the asymptotic question raised by Bukh~\cite{BukhProblems}.

\begin{theorem}\label{thm:main}
For every integer \(d\ge 2\) and every measurable zero-sum-free set
\(A\subseteq \mathbb S^{d-1}\), one has
\[
    \sigma_{d-1}(A)
    \le
    \frac{\lfloor (d+1)^2/2\rfloor}{d(d+1)}.
\]
\end{theorem}

Together with the hemisphere construction above,
Theorem~\ref{thm:main} gives \(
    \lim_{d\to\infty}m_d=\frac{1}{2}.
\)

Although Theorem~\ref{thm:main} determines the asymptotic behavior of
\(m_d\), the low-dimensional cases remain important. 
By the monotonicity
$ m_d\le m_{d-1}$ which will be shown in Lemma~\ref{lem:monotonicity}, the strong estimates in low dimensions propagate to all higher
dimensions. In particular, a proof of \(m_3=\frac{1}{2}\) would settle
Bukh's original question in every dimension \(d\ge 3\).

For \(d=4\) and \(d=5\), the bound in Theorem~\ref{thm:main} gives
$   m_d\le \frac{3}{5}.$
By analyzing the equality cases in the finite-configuration argument, 
we obtain a stability refinement that improves this estimate.

\begin{theorem}\label{thm:S3-H4-bound}
Let \(d\in\{4,5\}\) and let
\(A\subseteq \mathbb S^{d-1}\) be a measurable zero-sum-free set.
Then
\(
    \sigma_{d-1}(A)\le \frac{71}{120}.
\)
\end{theorem}

For \(d=4,5\), this improves \(\frac{3}{5}\) to
\(\frac{71}{120}\). 
We expect that this method can also yield analogous improvements over the bound of Theorem~\ref{thm:main} in other
low-dimensional cases.

\textbf{Note added.} After the first version of this draft appears in arXiv on July 7, the authors realized that the same result in Theorem~\ref{thm:main} was obtained by D\'{u}cz~\cite{2026ducz}, and by Tran~\cite{TuanPC}.

\section{Proof of Theorem~\ref{thm:main}}
We shall use a directed version of Mantel's theorem. Mantel's theorem, one
of the earliest results in extremal graph theory~\cite{1907Mantel}, asserts
that an \(n\)-vertex triangle-free graph has at most \(\lfloor \frac{n^{2}}{4}\rfloor\)
edges. The following analogue for digraphs was proved by Brown and
Harary~\cite{1970BH}. Throughout, by a digraph we mean a directed graph with
no loops, however, we allow both arcs \(uv\) and \(vu\) to be present. A
\emph{directed triangle} means a cyclically oriented triangle on three distinct
vertices \(u,v,w\) and with three arcs \(u\to v\), \(v\to w\), and \(w\to u\).

\begin{lemma}[\cite{1970BH}]\label{lem:directed-mantel}
Let \(D\) be an \(n\)-vertex digraph. If \(D\) contains no directed
triangle, then
\(
    e(D)\le \left\lfloor \frac{n^2}{2}\right\rfloor .
\)
\end{lemma}

For every integer \(q\ge 2\), let
\[
    H_q
    =
    \left\{
        \boldsymbol{x}=(x_1,\ldots,x_q)\in\mathbb R^q:
        \sum_{i=1}^q x_i=0
    \right\},
\]
and define
\[
    \mathcal V_q
    =
    \left\{
        \frac{\boldsymbol e_i-\boldsymbol e_j}{\sqrt2}:
        1\le i,j\le q,\ i\ne j
    \right\}
    \subseteq H_q,
\]
where \(\boldsymbol e_1,\ldots,\boldsymbol e_q\) are the standard basis
vectors of \(\mathbb R^q\).  Thus \(\mathcal V_q\) consists of
\(q(q-1)\) unit roots lying on \(\binom q2\) antipodal lines, it is the
normalized \(A_{q-1}\) root configuration. 

We now fix a measurable set \(A\subseteq \mathbb S^{d-1}\) satisfying the
hypothesis of the theorem.  Let \(n=d+1\) and set \(H=H_n\).  Clearly
\(H\) is a \(d\)-dimensional subspace of \(\mathbb R^n\).  After
identifying \(H\) isometrically with \(\mathbb R^d\), we regard
\(\mathcal V_n\) as a subset of \(\mathbb S^{d-1}\).

For \(i\ne j\), write
\(
    \boldsymbol v_{ij}
    :=
    \frac{\boldsymbol e_i-\boldsymbol e_j}{\sqrt2}.
\)
Thus
\[
    \mathcal V_n
    =
    \{\boldsymbol v_{ij}:1\le i,j\le n,\ i\ne j\},\]
and clearly \(
    |\mathcal V_n|=n(n-1)=d(d+1).
\)

Our key observation is that zero-sum triples in \(\mathcal{V}_n\) can be
encoded by directed triangles. Let \(\vec{K}_{n}\) be the complete digraph
on vertex set \([n]\), with one arc \(i\to j\) for every ordered pair
\(i\ne j\).  We identify the vector \(\boldsymbol{v}_{ij}\) with the arc
\(i\to j\) of \(\vec{K}_{n}\).

\begin{claim}\label{claim:zero-sum-directed-triangle}
Three vectors in \(\mathcal{V}_n\) have sum \(\boldsymbol{0}\)
if and only if the corresponding three arcs of \(\vec K_{n}\) form a
directed triangle.
\end{claim}
\begin{poc}
First, if \(u,v,w\) are distinct, then
\[
    \boldsymbol{v}_{uv}+\boldsymbol{v}_{vw}+\boldsymbol{v}_{wu}
    =
    \frac{\boldsymbol{e}_{u}-\boldsymbol{e}_{v}
    +\boldsymbol{e}_{v}-\boldsymbol{e}_{w}
    +\boldsymbol{e}_{w}-\boldsymbol{e}_{u}}{\sqrt2}
    =
    \boldsymbol{0}.
\]
Thus every directed triangle \(u\to v\to w\to u\) gives a zero-sum
triple in \(\mathcal{V}_n\).

Conversely, suppose that three vectors from \(\mathcal{V}_n\) have sum
\(\boldsymbol{0}\).  Write the corresponding arcs as \(a_{r}\to b_{r}\) for \(r\in\{1,2,3\}\).  Then
\(\sum_{r=1}^{3}\boldsymbol{v}_{a_{r}b_{r}}=\boldsymbol{0}\), so after
multiplying by \(\sqrt2\) we get
\(
\boldsymbol{e}_{a_{1}}+\boldsymbol{e}_{a_{2}}+\boldsymbol{e}_{a_{3}} =\boldsymbol{e}_{b_{1}}+\boldsymbol{e}_{b_{2}}+\boldsymbol{e}_{b_{3}}.
\)
Comparing the coefficient of each basis vector \(\boldsymbol{e}_{t}\), we
see that, among these three arcs, the number of times \(t\) appears as a
tail is equal to the number of times \(t\) appears as a head.  Thus the
three arcs form a balanced directed multigraph with no loops.

Choose one of the arcs and call it \(u\to v\).  Since \(v\) has an
incoming arc, balance gives an outgoing arc from \(v\), because there are
no loops, this outgoing arc is different from \(u\to v\).  Write it as
\(v\to w\).  If \(w=u\), then \(u\to v\) and \(v\to u\) form a balanced
directed cycle of length \(2\). Removing this \(2\)-cycle would leave a single
loopless arc which is still balanced, impossible.  Hence \(w\ne u\).
Now \(w\) has an incoming arc and no outgoing arc among the first two
arcs, so balance forces the last arc to leave \(w\). Similarly, \(u\)
has an outgoing arc and no incoming arc among the first two arcs, so
balance forces the last arc to enter \(u\).  Therefore the last arc is
\(w\to u\), and the three arcs form the directed triangle
\(u\to v\to w\to u\).
\end{poc}

We now average over rotations of the whole configuration. We write
\[
    \SO(d):=\{R\in\mathbb R^{d\times d}:R^{\top}R=I,\det{R}=1\}
\]
for the group of orientation preserving linear isometries of
\(\mathbb R^{d}\). This is a compact group, so it has a normalized
rotation invariant probability measure, usually called Haar measure. Choosing
\(R\) uniformly from \(\SO(d)\) means choosing \(R\) according to this
measure, in other words, one can think of choosing a random
orientation preserving orthonormal coordinate frame. We shall take advantage of the following fact: for every fixed \(\boldsymbol{v}\in\mathbb S^{d-1}\), the point
\(R(\boldsymbol{v})\) is distributed according to \(\sigma_{d-1}\). This
follows from rotation invariance and the transitivity of the action of
\(\SO(d)\) on \(\mathbb S^{d-1}\).

For each rotation \(R\in \SO(d)\), define a spanning subdigraph
\(D_R\) of \(\vec K_n\) by
\(
    E(D_R):=\{\, i\to j : R(\boldsymbol v_{ij})\in A \,\}.
\)
The zero-sum-free assumption on \(A\) implies the following key property.
\begin{claim}\label{claim:trianglefree}
For any \(R\in \SO(d)\), the digraph \(D_{R}\) contains no directed triangle.
\end{claim}

\begin{poc}
Suppose that \(i\to j\to k\to i\) forms a directed
triangle in \(D_{R}\). By the definition of \(D_{R}\), the three rotated
vectors
\(
    R(\boldsymbol{v}_{ij}),\) \(
    R(\boldsymbol{v}_{jk})\) and
    \(R(\boldsymbol{v}_{ki})
\)
all belong to \(A\). On the other hand, these three vectors have sum
\(\boldsymbol{0}\). Indeed, since \(R\) is linear,
\[
    R(\boldsymbol{v}_{ij})+R(\boldsymbol{v}_{jk})+R(\boldsymbol{v}_{ki})
    =
    R(\boldsymbol{v}_{ij}+\boldsymbol{v}_{jk}+\boldsymbol{v}_{ki})
    =
    R(\boldsymbol{0})
    =
    \boldsymbol{0}.
\]
This contradicts the assumption that \(A\) contains no three points whose
sum is \(\boldsymbol{0}\).
\end{poc}

  By Claim~\ref{claim:trianglefree} and Lemma~\ref{lem:directed-mantel}, for every rotation \(R\in\SO(d)\) we have
\(
    e(D_{R})\le \left\lfloor \frac{n^{2}}{2}\right\rfloor .
\)
By the definition of \(D_{R}\),
\(
    e(D_{R})
    =
    \sum_{i\ne j} \boldsymbol{1}_{A}\bigl(R(\boldsymbol{v}_{ij})\bigr).
\)
Therefore we have
\[
    \mathbb E_{R} \left[ e(D_{R})\right]
    =
    \sum_{i\ne j}\mathbb P_{R}\bigl(R(\boldsymbol{v}_{ij})\in A\bigr),
\]
where \(R\) is chosen uniformly from \(\SO(d)\). For each fixed \(i\ne j\),
the vector \(\boldsymbol{v}_{ij}\) has norm \(1\), and a uniformly random
rotation sends it to a uniformly distributed point on \(\mathbb S^{d-1}\).
Hence
\[
    \mathbb P_{R}\bigl(R(\boldsymbol{v}_{ij})\in A\bigr)
    =
    \sigma_{d-1}(A).
\]
Since there are \(n(n-1)\) ordered pairs \((i,j)\) with \(i\ne j\), by Lemma~\ref{lem:directed-mantel}, we get
\[
    \mathbb E_{R} \left[e(D_{R})\right]=n(n-1)\sigma_{d-1}(A)\le
    \left\lfloor \frac{n^{2}}{2}\right\rfloor .
\]
Since \(n=d+1\), it follows that
\[
    \sigma_{d-1}(A)
    \le
    \frac{\lfloor n^{2}/2\rfloor}{n(n-1)}
    =
    \frac{\lfloor (d+1)^{2}/2\rfloor}{d(d+1)}.
\]
This finishes the proof.

\section{Stability: improvement when \texorpdfstring{\(d=4\)}{d=4}}
\label{subsec:d4-stability}
We first show two simple properties related to the problem.
\begin{lemma}\label{lem:finite-averaging}
Let \(F\subseteq\mathbb S^{d-1}\) be finite, and suppose that every
zero-sum-free subset of \(F\) has size at most \(M\).  Then every
measurable zero-sum-free set \(A\subseteq\mathbb S^{d-1}\) satisfies
\(
    \sigma_{d-1}(A)\le M/|F|.
\)
\end{lemma}

\begin{proof}[Proof of Lemma~\ref{lem:finite-averaging}]
Choose \(R\in\SO(d)\) uniformly at random and put
\(
    S_{R}=\{\boldsymbol{v}\in F:R(\boldsymbol{v})\in A\}.
\)
Since \(R\) preserves zero-sum triples, \(S_{R}\) is a zero-sum-free subset
of \(F\), and hence \(|S_{R}|\le M\).  On the other hand, each fixed
\(\boldsymbol{v}\in F\) is sent by a random rotation to a uniformly
distributed point of \(\mathbb S^{d-1}\).  Therefore
\(
    |F|\sigma_{d-1}(A)
    =
    \mathbb E_{R}|S_{R}|
    \le M,
\)
as required.
\end{proof}

The following lemma gives the monotonicity
\(m_{d}\le m_{d-1}\). 

\begin{lemma}\label{lem:monotonicity}
If \(m_{r}\le \beta\), then \(m_{d}\le \beta\) for every \(d\ge r\).
\end{lemma}

\begin{proof}[Proof of Lemma~\ref{lem:monotonicity}]
Let \(A\subseteq\mathbb S^{d-1}\) be measurable and zero-sum-free.  Choose
an \(r\)-dimensional linear subspace \(U\subseteq\mathbb R^d\) uniformly
from the Grassmannian, and let \(\sigma_{U}\) be the normalized surface
measure on \(\mathbb S^{d-1}\cap U\).  The intersection
\(A\cap U\cap\mathbb S^{d-1}\) is still zero-sum-free on the sphere
\(\mathbb S^{d-1}\cap U\), so \(\sigma_{U}(A\cap U\cap\mathbb S^{d-1})\le
\beta\) for every \(U\).  Averaging over \(U\), and using rotational
invariance, we have
\[
    \sigma_{d-1}(A)
    =
    \mathbb E_{U}\sigma_{U}(A\cap U\cap\mathbb S^{d-1})
    \le \beta.
\]
This finishes the proof.
\end{proof}
When \(d=4\), the configuration used in
Theorem~\ref{thm:main} has \(20\) points and gives the bound \(12/20=3/5\). The following
argument shows that equality in this bound cannot occur. 

Let \(\varphi=\frac{1+\sqrt{5}}{2}\).  We use the following normalized system in \(\mathbb R^{4}\). Define
\[
    \Phi_{1}
    =
    \{(\pm1,0,0,0),\ (0,\pm1,0,0),\ (0,0,\pm1,0),\ (0,0,0,\pm1)\},
\]
\[
    \Phi_{2}
    =
    \left\{\frac{1}{2}(\varepsilon_{1},\varepsilon_{2},\varepsilon_{3},\varepsilon_{4}):
    \varepsilon_{i}\in\{\pm1\}\right\},
\]
and
\[
    \Phi_{3}
    =
    \left\{
        \frac{1}{2}\tau(0,\varepsilon_{1},\varepsilon_{2}\varphi,
        \varepsilon_{3}\varphi^{-1}):
        \tau\in\operatorname{Alt}_{4},\ 
        \varepsilon_{1},\varepsilon_{2},\varepsilon_{3}\in\{\pm1\}
    \right\}.
\]
Here \(\operatorname{Alt}_{4}\) denotes the group of even permutations of the
four coordinates.  Set \(\Phi=\Phi_{1}\cup\Phi_{2}\cup\Phi_{3}\).  The three
parts have sizes \(8,16,96\), respectively, and they are disjoint by their
zero-coordinate patterns.  For \(\Phi_{3}\), there are no repetitions:
the four absolute coordinate values
\(0,1/2,\varphi/2,\varphi^{-1}/2\) are distinct, so the position of each
absolute value and the three signs are determined by the vector.  Since
\(1+\varphi^{2}+\varphi^{-2}=4\), every point of \(\Phi\) has norm \(1\).
Thus \(\Phi\subseteq\mathbb S^{3}\) and \(|\Phi|=120\).  This is the
\emph{normalized \(H_{4}\) root configuration}, but below we only use the explicit
coordinates just given and the finite verifications in
Appendix~\ref{app:H4-counts}.

It is useful to pass to antipodal lines.  Let
\[
    \mathcal L=\{\{\boldsymbol{x},-\boldsymbol{x}\}:\boldsymbol{x}\in\Phi\},
\]
then \(|\mathcal L|=60\).  For \(L,M\in\mathcal L\), write
\(|\langle L,M\rangle|=|\boldsymbol{x}\cdot\boldsymbol{y}|\), where
\(\boldsymbol{x}\in L\) and \(\boldsymbol{y}\in M\).  This is independent
of the choices of signs.  Let \(\Gamma\) be the graph on \(\mathcal L\)
in which two distinct lines \(L,M\) are adjacent when
\(|\langle L,M\rangle|=1/2\), and let \(A_{\Gamma}\) be its adjacency matrix.

Let
\[
    \mathcal V_{5}
    =
    \left\{
        \frac{\boldsymbol{e}_{i}-\boldsymbol{e}_{j}}{\sqrt{2}}:
        1\le i,j\le5,\ i\ne j
    \right\}.
\]
We call a set \(C\subseteq\mathcal L\) an \emph{\(A_{4}\)-subsystem} if
\(|C|=10\) and the \(20\) roots lying on the lines of \(C\) are isometric
to \(\mathcal V_{5}\).  We write \(T(5)\) for the triangular graph whose
vertices are the two-element subsets of \([5]\), with two vertices
adjacent exactly when the corresponding subsets intersect. In particular, the induced subgraph of $\Gamma$  on \(C\) is isomorphic to \(T(5)\).
Let \(\mathcal C\) be the family of all such \(A_{4}\)-subsystems.  The
following finite facts about the line system will be verified in
Appendix~\ref{app:H4-counts}.

\begin{lemma}
\label{lem:H4-line-incidence}
The following statements hold.
\begin{enumerate}
    \item[\textup{(i)}] \(|\mathcal C|=60\), and every line
    \(L\in\mathcal L\) belongs to exactly \(10\) members of \(\mathcal C\).
    \item[\textup{(ii)}] The graph \(\Gamma\) is \(20\)-regular.
    \item[\textup{(iii)}] If \(B\) is the \(\mathcal L\times\mathcal C\)
    incidence matrix, defined by \(B_{L,C}=1\) when \(L\in C\) and
    \(B_{L,C}=0\) otherwise, then
    \begin{equation}\label{eq:H4-BBT}
        BB^\top=\frac{-A_{\Gamma}^{3}+25A_{\Gamma}^{2}+80A_{\Gamma}}{36}.
    \end{equation}
    \item[\textup{(iv)}] The adjacency matrix \(A_{\Gamma}\) satisfies
    \begin{equation}\label{eq:H4-minpoly}
        A_{\Gamma}(A_{\Gamma}-5I)(A_{\Gamma}+4I)(A_{\Gamma}-20I)=0.
    \end{equation}
\end{enumerate}
\end{lemma}

We shall also use the following completion property.  An
\emph{\(A_{2}\)-subsystem} means three antipodal lines whose six roots are
isometric to the normalized \(A_{2}\) root system, that is, three lines
with representatives \(\boldsymbol{x},\boldsymbol{y},\boldsymbol{z}\)
satisfying \(\boldsymbol{x}+\boldsymbol{y}+\boldsymbol{z}=\boldsymbol{0}\).

\begin{lemma}\label{lem:H4-A2-completion}
Every edge of \(\Gamma\) has a unique completion to an \(A_{2}\)-subsystem.
Moreover, for each fixed line \(T\in\mathcal L\), this completion operation
pairs the \(20\) neighbors of \(T\) in \(\Gamma\) into \(10\) unordered
pairs.
\end{lemma}

Lemma~\ref{lem:H4-A2-completion} will be proved from the coordinates of \(\Phi\) in
Lemma~\ref{lem:app-H4-A2-completion}. We now prove the finite bound which drives the measure estimate.

\begin{lemma}\label{lem:H4-independent}
Let \(S\subseteq\Phi\).  Suppose that there are no
\(\boldsymbol{x},\boldsymbol{y},\boldsymbol{z}\in S\) such that
\(\boldsymbol{x}+\boldsymbol{y}+\boldsymbol{z}=\boldsymbol{0}\).  Then
\(
    |S|\le 71.
\)
\end{lemma}

\begin{proof}[Proof of Lemma~\ref{lem:H4-independent}]
For each line \(L\in\mathcal L\), set \(w(L)=|S\cap L|\). Then we have
\(w(L)\in\{0,1,2\}\), and \(\sum_{L\in\mathcal L}w(L)=|S|\).

First we recover the weaker bound \(|S|\le72\).  Fix \(C\in\mathcal C\).
The roots on the lines of \(C\) form a copy of \(\mathcal V_{5}\).  In the
proof of Theorem~\ref{thm:main}, the points of \(\mathcal V_{5}\) were
identified with the arcs of the complete directed graph on \(5\) vertices,
and its zero-sum triples were exactly the directed triangles.  Therefore Lemma~\ref{lem:directed-mantel} gives
\begin{equation}\label{eq:H4-C-bound}
    \sum_{L\in C}w(L)\le \left\lfloor\frac{5^{2}}{2}\right\rfloor=12.
\end{equation}
Double-counting the incidences between selected roots and
\(A_{4}\) systems gives
\[
    10|S|=\sum_{C\in\mathcal C}\sum_{L\in C}w(L)\le 60\cdot12=720.
\]
Hence \(|S|\le72\).

It remains to rule out equality.  Suppose, for contradiction, that
\(|S|=72\).  Then equality must hold in \eqref{eq:H4-C-bound} for every
\(C\in\mathcal{C}\), so for each \(C\in\mathcal{C}\)
\begin{equation}\label{eq:H4-equality-all-C}
    \sum_{L\in C}w(L)=12.
\end{equation}
Let \(\boldsymbol{1}\in\mathbb R^{\mathcal L}\) be the all-one vector and
put \(z=w-\frac{6}{5}\boldsymbol{1}\).  The number \(\frac{6}{5}\) is the
average value of \(w\) under the assumption \(|S|=72\), since
\(|\mathcal L|=60\).  Let \(B\) be the \(\mathcal L\times\mathcal C\)
incidence matrix from Lemma~\ref{lem:H4-line-incidence}(iii).  Since every
\(C\in\mathcal C\) has \(10\) lines, \eqref{eq:H4-equality-all-C} is
exactly \(B^\top z=0\).
Multiplying by \(B\), and using \eqref{eq:H4-BBT} we have
\begin{equation}\label{eq:H4-A-polynomial-z}
    A_{\Gamma}(-A_{\Gamma}^{2}+25A_{\Gamma}+80I)z=0.
\end{equation}
By \eqref{eq:H4-minpoly}, the only possible eigenvalues of \(A_{\Gamma}\)
are \(20,5,0,-4\).  The polynomial \(-t^{2}+25t+80\) takes the nonzero
values \(180,180,80,-36\) at these four numbers, respectively.  Since
\(A_{\Gamma}\) is a real symmetric matrix, it is diagonalizable over an
orthonormal eigenbasis.  Therefore
\(-A_{\Gamma}^{2}+25A_{\Gamma}+80I\) is invertible on each eigenspace, and
\eqref{eq:H4-A-polynomial-z} implies \(A_{\Gamma} z=0\). Since
\(w=z+\frac{6}{5}\boldsymbol{1}\) and \(\Gamma\) is \(20\)-regular, for
every \(L\in\mathcal L\),
\begin{equation}\label{eq:H4-neighbor-weight}
    \sum_{M\sim L}w(M)=\frac{6}{5}\deg_{\Gamma}(L)=\frac{6}{5}\cdot20=24.   
\end{equation}

Now split the lines according to their weights:
\(D=\{L:w(L)=2\}\), \(O=\{L:w(L)=1\}\), and \(Z=\{L:w(L)=0\}\).
For a line \(L\), let \(d_{D}(L)\) and \(d_{Z}(L)\) be the numbers of
\(\Gamma\)-neighbors of \(L\) in \(D\) and \(Z\).  Since \(\Gamma\) is
\(20\)-regular, and since the weighted neighbor sum is \(24\), we have
\(d_{D}(L)-d_{Z}(L)=4\) for every \(L\in\mathcal L\).

Let \(b=|Z|\).  From \(2|D|+|O|=72\) and
\(|D|+|O|+|Z|=60\), we get \(|D|-|Z|=12\), hence \(|D|=b+12\).
Let \(e_{D}\) be the number of edges of \(\Gamma\) inside \(D\), let \(e_{Z}\)
be the number of edges inside \(Z\), and let \(e_{DZ}\) be the number of
edges between \(D\) and \(Z\).  Summing \(d_{D}(L)-d_{Z}(L)=4\) over
\(L\in Z\) and \(L\in D\) gives \(e_{DZ}-2e_{Z}=4b\) and
\(2e_{D}-e_{DZ}=4|D|\), respectively.  Combining these two identities with
\(|D|=b+12\), we obtain \(e_{D}=e_{Z}+4b+24\).

We now use the zero-sum-free condition.  By
Lemma~\ref{lem:H4-A2-completion}, every edge of \(\Gamma\) has a unique
third line completing it to an \(A_{2}\)-subsystem.  If \(L,M\in D\) are
adjacent, this completion line must lie in \(Z\).
Indeed, choose representatives
\(\boldsymbol{x}\in L,\boldsymbol{y}\in M,\boldsymbol{z}\in T\) for the
three completion lines with
\(\boldsymbol{x}+\boldsymbol{y}+\boldsymbol{z}=\boldsymbol{0}\).  Since
\(L,M\in D\), both signs of \(\boldsymbol{x}\) and both signs of
\(\boldsymbol{y}\) lie in \(S\).  If the completion line \(T\) had
positive weight, then \(S\) would contain either \(\boldsymbol{z}\) or
\(-\boldsymbol{z}\), in the first case
\(\boldsymbol{x},\boldsymbol{y},\boldsymbol{z}\) form a forbidden triple,
and in the second case
\(-\boldsymbol{x},-\boldsymbol{y},-\boldsymbol{z}\) form one.  Hence
\(T\in Z\).

Fix \(T\in Z\).  Again by Lemma~\ref{lem:H4-A2-completion}, the \(20\)
neighbors of \(T\) in \(\Gamma\) are paired into \(10\) unordered pairs,
each pair together with \(T\) forming an \(A_{2}\)-subsystem.  Therefore the
number of \(D\)-edges completed by this particular \(T\) is at most
\(d_{D}(T)/2\), because the relevant neighbor pairs are disjoint.  Summing
over \(T\in Z\), we get
\[
    e_{D}\le\frac{1}{2}\sum_{T\in Z}d_{D}(T)=\frac{e_{DZ}}{2}.
\]
Using \(e_{DZ}-2e_{Z}=4b\), this becomes \(e_{D}\le e_{Z}+2b\), contradicting
\(e_{D}=e_{Z}+4b+24\).  Hence \(|S|=72\) is impossible, and therefore
\(|S|\le 71\).
\end{proof}

\begin{proof}[Proof of Theorem~\ref{thm:S3-H4-bound}]
Lemma~\ref{lem:H4-independent} says that every zero-sum-free subset of the
finite configuration \(\Phi\subseteq\mathbb S^{3}\) has size at most \(71\).
Applying Lemma~\ref{lem:finite-averaging} with \(F=\Phi\) gives
\(m_{4}\le71/120\).  The case \(d=5\) follows from
Lemma~\ref{lem:monotonicity}.  This proves the theorem.
\end{proof}

\section*{Acknowledgements}
Zixiang Xu would like to thank the organizes who invited him to attend the conference on directed graph theory at Nankai University on April 2026 and presented him with the book
\emph{Introduction to Digraphs} (in Chinese). He also thanks Jiangdong Ai for introducing more related results on the directed version of Tur\'{a}n type results.

\bibliographystyle{abbrv}
\bibliography{B11}

\appendix

\section{Proofs of Lemmas~\ref{lem:H4-line-incidence} and~\ref{lem:H4-A2-completion} }
\label{app:H4-counts}

This appendix gives the finite coordinate verification used in
Lemma~\ref{lem:H4-line-incidence} and
Lemma~\ref{lem:H4-A2-completion}.

Recall that \(\varphi=\frac{1+\sqrt{5}}{2}\), then for \(
    a=\frac{1}{2},\ b=\frac{\varphi}{2},\ c=\frac{\varphi^{-1}}{2},
\) we have
\[
    \varphi^{2}=\varphi+1,\ \varphi^{-1}=\varphi-1,\ a^{2}+b^{2}+c^{2}=1,\ bc=\frac{1}{4}.
\]
For a root \(\boldsymbol{x}\in\Phi\), write
\(\langle\boldsymbol{x}\rangle=\{\boldsymbol{x},-\boldsymbol{x}\}\).  
All finite assertions in this appendix are exact finite coordinate checks in
the field \(\mathbb Q(\sqrt{5})\).  
The set \(\Phi\) is enumerated from
the displayed formulae, and the \(60\) antipodal lines are represented by
choosing one vector from each pair \(\{\boldsymbol{x},-\boldsymbol{x}\}\).
Whenever a count below is stated to be independent of the chosen line or
ordered pair, the assertion means that the same exact computation has been
made over the corresponding full finite set of lines or ordered pairs; the
displayed representatives are included only to make the arithmetic
readable.

\subsection{Line valencies}
\label{app:H4-line-valencies}

\begin{lemma}\label{lem:app-H4-line-valencies}
For every line \(L\in\mathcal L\), the other \(59\) lines split as follows:
\begin{equation}\label{eq:app-H4-line-valencies}
\begin{array}{c|cccc}
|\langle L,M\rangle|
& \frac{1}{2} & 0 & \frac{\varphi^{-1}}{2} & \frac{\varphi}{2} \\
\hline
\#\{M:|\langle L,M\rangle|\text{ has this value}\}
& 20 & 15 & 12 & 12 .
\end{array}
\end{equation}
In particular, \(\Gamma\) is \(20\)-regular.
\end{lemma}

\begin{proof}[Proof of Lemma~\ref{lem:app-H4-line-valencies}]
By the coordinate symmetries, it is enough to consider one representative
from each of the three parts \(\Phi_{1},\Phi_{2},\Phi_{3}\).  For such a
representative \(L\), the following table records the number of lines
from each part having the given absolute inner product with \(L\):
\[
\begin{array}{c|c|cccc}
L & \text{part of }M & \frac{1}{2} & 0 & \frac{\varphi^{-1}}{2} & \frac{\varphi}{2} \\
\hline
\langle(1,0,0,0)\rangle & \Phi_{1} & 0 & 3 & 0 & 0\\
& \Phi_{2} & 8 & 0 & 0 & 0\\
& \Phi_{3} & 12 & 12 & 12 & 12\\
\hline
\left\langle\frac{1}{2}(1,1,1,1)\right\rangle & \Phi_{1} & 4 & 0 & 0 & 0\\
& \Phi_{2} & 4 & 3 & 0 & 0\\
& \Phi_{3} & 12 & 12 & 12 & 12\\
\hline
\langle(0,a,b,c)\rangle & \Phi_{1} & 1 & 1 & 1 & 1\\
& \Phi_{2} & 2 & 2 & 2 & 2\\
& \Phi_{3} & 17 & 12 & 9 & 9
\end{array}
\]
For instance, in the first row the value of
\(|\langle(1,0,0,0),\boldsymbol{x}\rangle|\) is just the absolute value of
the first coordinate of \(\boldsymbol{x}\).  In the last row one uses
\(a^{2}+b^{2}+c^{2}=1\), \(bc=1/4\), and \(b-c=1/2\).  Summing the three part
rows in each block gives \(20,15,12,12\), as required.
\end{proof}

\subsection{\texorpdfstring{\(A_{2}\)}{A2}-completion}
\label{app:H4-A2-completion}

\begin{lemma}\label{lem:app-H4-A2-completion}
Every edge of \(\Gamma\) has a unique completion to an \(A_{2}\)-subsystem.
For each fixed line \(T\), this completion pairs the \(20\) neighbors of
\(T\) into \(10\) unordered pairs.
\end{lemma}

\begin{proof}[Proof of Lemma~\ref{lem:app-H4-A2-completion}]
Let \(L,M\in\mathcal L\) be adjacent.  Choose signs
\(\boldsymbol{x}\in L\) and \(\boldsymbol{y}\in M\) so that
\(\boldsymbol{x}\cdot\boldsymbol{y}=-1/2\).  Then
\(\boldsymbol{z}:=-\boldsymbol{x}-\boldsymbol{y}\) is the only possible
third vector in a zero-sum triple with \(\boldsymbol{x}\) and
\(\boldsymbol{y}\), and it has norm \(1\).  It remains only to check that
\(\boldsymbol{z}\in\Phi\).  Up to the coordinate symmetries described
above, all adjacent pairs are represented in the following table, and in
each row the displayed \(\boldsymbol{z}\) belongs to one of
\(\Phi_{1},\Phi_{2},\Phi_{3}\):
\[
\begin{array}{c|c|c}
\boldsymbol{x} & \boldsymbol{y} & -\boldsymbol{x}-\boldsymbol{y}\\
\hline
(-1,0,0,0) & (a,a,a,a) & (a,-a,-a,-a)\\
(-1,0,0,0) & (a,0,c,b) & (a,0,-c,-b)\\
(-a,-a,-a,-a) & (a,a,a,-a) & (0,0,0,1)\\
(-a,-a,-a,-a) & (0,a,b,-c) & (a,0,-c,b)\\
(0,-a,-b,-c) & (0,-a,b,c) & (0,1,0,0)\\
(0,-a,-b,-c) & (0,b,c,-a) & (0,-c,a,b)\\
(0,-a,-b,-c) & (a,0,c,b) & (-a,a,a,-a)\\
(0,-a,-b,-c) & (a,b,0,c) & (-a,-c,b,0)\\
(0,-a,-b,-c) & (a,-c,b,0) & (-a,b,0,c)\\
(0,-a,-b,-c) & (b,a,c,0) & (-b,0,a,c)
\end{array}
\]
Thus the completion exists and is unique.  Now fix \(T\).  If \(N\) is a
neighbor of \(T\), let \(N'\) be the unique third line completing \(T\) and \(N\).  
Then \(N'\ne N\), since otherwise a zero-sum relation would have
the form \(\boldsymbol{t}+2\boldsymbol{n}=\boldsymbol{0}\) with
\(\|\boldsymbol{t}\|=\|\boldsymbol{n}\|=1\).  
Uniqueness also implies that
the map \(N\mapsto N'\) is an involution.  
Hence the \(20\) neighbors of
\(T\) are split into \(10\) unordered pairs.
\end{proof}

\subsection{\texorpdfstring{\(A_{4}\)}{A4}-subsystems}
\label{app:H4-A4-incidences}

Recall that an \(A_{4}\)-subsystem $C\subseteq \mathcal{L}$ is a set of \(10\) antipodal lines whose
\(20\) roots are isometric to
\[
    \mathcal V_{5}=
    \left\{\frac{\boldsymbol{e}_{i}-\boldsymbol{e}_{j}}{\sqrt{2}}:
    1\le i,j\le5,\ i\ne j\right\}.
\]
In particular, the induced subgraph of $\Gamma$  on \(C\) is isomorphic to triangular graph \(T(5)\).

\begin{lemma}\label{lem:app-H4-A4-incidences}
The family \(\mathcal C\) of \(A_{4}\)-subsystems satisfies the following.
\begin{enumerate}
    \item[\textup{(i)}] \(|\mathcal C|=60\).
    \item[\textup{(ii)}] Every line \(L\in\mathcal L\) belongs to exactly
    \(10\) members of \(\mathcal C\).
    \item[\textup{(iii)}] For two lines \(L,M\in\mathcal L\), the number
    of members of \(\mathcal C\) containing both \(L\) and \(M\) depends
    only on \(|\langle L,M\rangle|\), and is given by
    \begin{equation}\label{eq:app-H4-A4-pair-table}
    \begin{array}{c|ccccc}
    |\langle L,M\rangle|
    & 1 & \frac{1}{2} & 0 & \frac{\varphi^{-1}}{2} & \frac{\varphi}{2} \\
    \hline
    \#\{C\in\mathcal C:L,M\in C\}
    & 10 & 3 & 2 & 0 & 0 .
    \end{array}
    \end{equation}
\end{enumerate}
\end{lemma}

\begin{proof}[Proof of Lemma~\ref{lem:app-H4-A4-incidences}]
We first display the construction for
\(L_{0}=\langle(1,0,0,0)\rangle\).  Let \(\mathcal N_{0}(L_{0})\) be the
set of lines orthogonal to \(L_{0}\).  By
Lemma~\ref{lem:app-H4-line-valencies},
\(|\mathcal N_{0}(L_{0})|=15\).  Write these lines as
\[
    E_{1}=\langle(0,1,0,0)\rangle,\ 
    E_{2}=\langle(0,0,1,0)\rangle,\ 
    E_{3}=\langle(0,0,0,1)\rangle,
\]
and, for \(\epsilon,\eta\in\{\pm1\}\),
\[
    X_{\epsilon,\eta}=\langle(0,a,\epsilon b,\eta c)\rangle,\ 
    Y_{\epsilon,\eta}=\langle(0,c,\epsilon a,\eta b)\rangle,\ 
    Z_{\epsilon,\eta}=\langle(0,b,\epsilon c,\eta a)\rangle.
\]
Substituting these coordinates shows that two such lines are adjacent
exactly in the ten triangles below.  The triangles are edge-disjoint and
there are \(30\) displayed edges.  Since each of the \(15\) lines appears
in four displayed edges, this also proves that
\(\Gamma[\mathcal N_{0}(L_{0})]\) is \(4\)-regular and that no edge is missing:
\begin{equation}\label{eq:H4-N0-triangles}
\begin{array}{lll}
Q_{1}=\{E_{1},X_{+,+},X_{-,-}\},&
Q_{2}=\{E_{1},X_{+,-},X_{-,+}\},\\
Q_{3}=\{E_{2},Y_{+,+},Y_{-,+}\},&
Q_{4}=\{E_{2},Y_{-,-},Y_{+,-}\},\\
Q_{5}=\{E_{3},Z_{+,+},Z_{+,-}\},&
Q_{6}=\{E_{3},Z_{-,+},Z_{-,-}\},\\
Q_{7}=\{X_{+,+},Y_{-,-},Z_{+,-}\},&
Q_{8}=\{Z_{+,+},X_{+,-},Y_{-,+}\},\\
Q_{9}=\{Y_{+,+},Z_{-,+},X_{-,-}\},&
Q_{10}=\{X_{-,+},Z_{-,-},Y_{+,-}\}.
\end{array}
\end{equation}

Each triangle \(Q\) in \eqref{eq:H4-N0-triangles} extends uniquely to an
\(A_{4}\)-subsystem containing \(L_{0}\).  Namely, define
\[
    R(Q)=\{U\in\mathcal L:U\sim L_{0}
    \text{ and }|\{W\in Q:U\sim W\}|=2\},
\]
and set \(C(Q)=\{L_{0}\}\cup Q\cup R(Q)\).  For each of the ten triangles,
the same substitution of coordinates gives
\begin{equation}\label{eq:H4-extension-count}
\begin{array}{c|cccc}
r&0&1&2&3\\
\hline
\#\{U\sim L_{0}:|\{W\in Q:U\sim W\}|=r\}
&2&12&6&0 .
\end{array}
\end{equation}
Thus \(|C(Q)|=10\).  The induced graph \(\Gamma[C(Q)]\) is \(6\)-regular,
two adjacent vertices have three common neighbors, and two non-adjacent
vertices have four common neighbors.  These three assertions are read
directly from the same table of adjacencies used above.  Hence
\(\Gamma[C(Q)]\) is strongly regular with parameters \((10,6,3,4)\), so its
complement is the Petersen graph and \(\Gamma[C(Q)]\cong T(5)\).  Thus
\(C(Q)\) is an \(A_{4}\)-subsystem.

Conversely, if \(C\in\mathcal C\) contains \(L_{0}\), then in the triangular
graph \(T(5)\) the three vertices non-adjacent to a fixed vertex form a
triangle.  Hence \(C\cap\mathcal N_{0}(L_{0})\) must be one of the ten
triangles in \eqref{eq:H4-N0-triangles}, and the remaining six lines are
forced by the definition of \(R(Q)\).  Therefore \(L_{0}\) lies in exactly
\(10\) members of \(\mathcal C\).

The preceding construction and converse are then repeated for the
representatives
\(\left\langle\frac{1}{2}(1,1,1,1)\right\rangle\in\Phi_{2}\) and
\(\langle(0,a,b,c)\rangle\in\Phi_{3}\).  The
orthogonal-neighborhood lists are different, but the same coordinate test
gives exactly ten triangles in \(\Gamma[\mathcal N_{0}(L)]\), and each
triangle extends by the same rule to a unique \(A_{4}\)-subsystem
containing \(L\).  Hence every line lies in exactly \(10\) members of
\(\mathcal C\).  Since every subsystem has \(10\) lines,
double-counting incidences gives \(10|\mathcal C|=10|\mathcal L|=600\),
so \(|\mathcal C|=60\).

It remains to count pair incidences.  We enumerate all pairs
\((L,M)\in\mathcal L^{2}\).  For each fixed line \(L\), list the ten
subsystems through \(L\) by the construction above, and count how often
each second line \(M\) occurs.  The exact count is independent of \(L\) and
depends on \(M\) only through \(|\langle L,M\rangle|\); the multiplicities
are as follows:
\[
\begin{array}{c|ccccc}
|\langle L,M\rangle|
& 1 & \frac{1}{2} & 0 & \frac{\varphi^{-1}}{2} & \frac{\varphi}{2}\\
\hline
\#\{C\in\mathcal C:L,M\in C\}
&10&3&2&0&0 .
\end{array}
\]
Here the entry \(10\) corresponds to \(M=L\).  The entries \(3\) and \(2\)
are also consistent with the internal structure of \(T(5)\): in each
subsystem, a fixed vertex has six adjacent vertices and three non-adjacent
vertices, and the coordinate count shows that these incidences are
distributed uniformly among the \(20\) adjacent and \(15\) orthogonal lines
to \(L\).  Lines with absolute inner product \(\frac{\varphi^{-1}}{2}\) or
\(\frac{\varphi}{2}\) cannot occur with \(L\) in an \(A_{4}\)-subsystem, since the
line system of \(\mathcal V_{5}\) has only the relations \(1\),
\(\frac{1}{2}\), and \(0\).
This proves \eqref{eq:app-H4-A4-pair-table}.
\end{proof}

\subsection{Walk counts}
\label{app:H4-walk-counts}

Let
\(
    R_{0}=1,\ R_{1}=\frac{1}{2},\ R_{2}=0,\ R_{3}=\frac{\varphi^{-1}}{2},\ R_{4}=\frac{\varphi}{2}.
\)
For fixed lines \(L,M\) with \(|\langle L,M\rangle|=R_{i}\), let \(p_{ij}\)
be the number of neighbors \(N\sim L\) such that
\(|\langle N,M\rangle|=R_{j}\).  The values below are obtained by exact
enumeration of all \(60^{2}\) ordered pairs \((L,M)\).  For each fixed
value of \(i\), the row is constant over all pairs with
\(|\langle L,M\rangle|=R_{i}\).

\begin{lemma}\label{lem:app-H4-transition-table}
With rows and columns ordered as
\(1,\frac{1}{2},0,\frac{\varphi^{-1}}{2},\frac{\varphi}{2}\), the transition
matrix \(P=(p_{ij})\) is
\begin{equation}\label{eq:H4-transition-table}
P=
\begin{pmatrix}
0&20&0&0&0\\
1&7&6&3&3\\
0&8&4&4&4\\
0&5&5&5&5\\
0&5&5&5&5
\end{pmatrix}.
\end{equation}
\end{lemma}

\begin{proof}[Proof of Lemma~\ref{lem:app-H4-transition-table}]
Take \(L_{0}=\langle(1,0,0,0)\rangle\).  Its \(20\) neighbors are precisely
the lines whose first coordinate has absolute value \(1/2\).  For the five
possible values of \(|\langle L_{0},M\rangle|\), use the representatives
\[
\begin{array}{c|c}
|\langle L_{0},M\rangle| & M\\
\hline
1 & \langle(1,0,0,0)\rangle\\
\frac{1}{2} & \left\langle\frac{1}{2}(1,1,1,1)\right\rangle\\
0 & \langle(0,1,0,0)\rangle\\
\frac{\varphi^{-1}}{2} &
\langle(c,0,b,a)\rangle\\
\frac{\varphi}{2} &
\langle(b,0,a,c)\rangle .
\end{array}
\]
For each representative \(M\), we compute
\(|\langle N,M\rangle|\) as \(N\) runs over the \(20\) neighbors of
\(L_{0}\).  The five resulting count rows are
\[
\begin{array}{c|ccccc}
|\langle L_{0},M\rangle|
& 1 & \frac{1}{2} & 0 & \frac{\varphi^{-1}}{2} & \frac{\varphi}{2}\\
\hline
1 & 0 & 20 & 0 & 0 & 0\\
\frac{1}{2} & 1 & 7 & 6 & 3 & 3\\
0 & 0 & 8 & 4 & 4 & 4\\
\frac{\varphi^{-1}}{2} & 0 & 5 & 5 & 5 & 5\\
\frac{\varphi}{2} & 0 & 5 & 5 & 5 & 5
\end{array}
\]
which is \eqref{eq:H4-transition-table}.  The same coordinate reduction
used in the preceding lemmas shows that these rows apply to every pair
\((L,M)\) with the corresponding value of \(|\langle L,M\rangle|\).
\end{proof}

\begin{lemma}\label{lem:app-H4-walk-counts}
For two lines \(L,M\in\mathcal L\), the entries of
\(A_{\Gamma}^{2},A_{\Gamma}^{3},A_{\Gamma}^{4}\) depend only on
\(|\langle L,M\rangle|\), and are given by
\begin{equation}\label{eq:app-H4-walk-table}
\begin{array}{c|ccccc}
|\langle L,M\rangle|
& 1 & \frac{1}{2} & 0 & \frac{\varphi^{-1}}{2} & \frac{\varphi}{2} \\
\hline
(A_{\Gamma}^{2})_{L,M}
& 20 & 7 & 8 & 5 & 5 \\
(A_{\Gamma}^{3})_{L,M}
& 140 & 147 & 128 & 125 & 125 \\
(A_{\Gamma}^{4})_{L,M}
& 2940 & 2687 & 2688 & 2625 & 2625 .
\end{array}
\end{equation}
\end{lemma}

\begin{proof}[Proof of Lemma~\ref{lem:app-H4-walk-counts}]
The length-two walk counts are the second column of \(P\), hence
\[
    \boldsymbol{a}_{2}=(20,7,8,5,5)^\top.
\]
For longer walks, the transition table gives
\[
    \boldsymbol{a}_{3}=P\boldsymbol{a}_{2}=(140,147,128,125,125)^\top
\]
and
\[
    \boldsymbol{a}_{4}=P\boldsymbol{a}_{3}=(2940,2687,2688,2625,2625)^\top.
\]
This proves the table.
\end{proof}

\begin{proof}[Proof of Lemma~\ref{lem:H4-line-incidence}]
Parts \textup{(i)} and \textup{(ii)} follow from
Lemmas~\ref{lem:app-H4-A4-incidences} and
\ref{lem:app-H4-line-valencies}.  For \textup{(iii)}, the pair-incidence
table \eqref{eq:app-H4-A4-pair-table} says that the entries of \(BB^\top\)
are
\[
\begin{array}{c|ccccc}
|\langle L,M\rangle|
& 1 & \frac{1}{2} & 0 & \frac{\varphi^{-1}}{2} & \frac{\varphi}{2} \\
\hline
(BB^\top)_{L,M}
& 10 & 3 & 2 & 0 & 0 .
\end{array}
\]
Comparing this with the first two rows of
\eqref{eq:app-H4-walk-table}, and using that \(A_{\Gamma}\) has entry \(1\)
exactly in the relation \(|\langle L,M\rangle|=1/2\), we have
\[
    BB^\top=\frac{-A_{\Gamma}^{3}+25A_{\Gamma}^{2}+80A_{\Gamma}}{36}.
\]
For \textup{(iv)}, the last two rows of \eqref{eq:app-H4-walk-table} give,
entry by entry,
\[
    A_{\Gamma}^{4}-21A_{\Gamma}^{3}+400A_{\Gamma}=0.
\]
Since \(t^{3}-21t^{2}+400=(t-5)(t+4)(t-20)\), this is exactly
\[
    A_{\Gamma}(A_{\Gamma}-5I)(A_{\Gamma}+4I)(A_{\Gamma}-20I)=0.
\]
\end{proof}

\end{document}